\documentclass[11pt]{amsart}
\usepackage{amsfonts}
\usepackage{amsmath}
\usepackage{graphicx}
\usepackage{subfig}
\usepackage{url}

\theoremstyle{plain}
\newtheorem{theorem}{Theorem}
\newtheorem{proposition}[theorem]{Proposition}
\newtheorem{lemma}[theorem]{Lemma}
\newtheorem{corollary}[theorem]{Corollary}
\newtheorem{conjecture}[theorem]{Conjecture}

\theoremstyle{definition}

\theoremstyle{remark}
\newtheorem{remark}[theorem]{Remark}

\DeclareMathOperator{\rank}{rank}

\title[New Counterexamples to Hendrickson's Conjecture]{New Classes of Counterexamples to Hendrickson's Global Rigidity Conjecture}
\thanks{Special thanks to Dylan Thurston, Joe Ross, and Ina Petkova for their immeasurable help. This work was partially supported by NSF RTG Grant 07-39392.}
\author{Samuel Frank}
\address{Department of Mathematics, Columbia University, New York, NY 10027}
\email{smf2147@columbia.edu}
\author{Jiayang Jiang}
\address{Department of Mathematics, Columbia University, New York, NY 10027}
\email{jj2333@columbia.edu}
\begin{document}

\maketitle

\begin{abstract}
We examine the generic local and global rigidity of various graphs in $\mathbb{R}^d$.  Bruce Hendrickson showed that some necessary conditions for generic global rigidity are ($d+1$)-connectedness and generic redundant rigidity and hypothesized that they were sufficient in all dimensions.  We analyze two classes of graphs that satisfy Hendrickson's conditions for generic global rigidity, yet fail to be generically globally rigid.  We find a large family of bipartite graphs for $d > 3$, and we define a construction that generates infinitely many graphs in $\mathbb{R}^5$.  Finally, we state some conjectures for further exploration.
\end{abstract}

\begin{section}{Introduction and Preliminaries}
A \emph{framework} consists of a graph whose vertices have been assigned coordinates in $\mathbb{R}^d$.  An important  question is whether or not a given framework is \emph{locally rigid}, that is, whether there is a way to continuously deform the framework while maintaing its edge lengths.  A related question is whether or not the framework is \emph{globally rigid}, or whether any other framework with the same underlying graph and the same edge lengths is equivalent up to Euclidean motions (combinations of reflections, rotations, and translations).  For $d \leq 3$, this question has many important real-world applications, such as analyzing the structural integrity of buildings or determining molecular structure.  However, the problem is not fully understood, and only recently has it been explored in great detail.  Some complete bipartite graphs have the characteristic that most of their frameworks are not globally rigid, but in a non-obvious way; these graphs have been well characterized by Connelly \cite{cn}, as well as Bolker and Roth \cite{br}.  In this paper, we present more graphs with this characteristic.

A graph is defined by $G = (V, E)$ with $|V| = v$ and $|E| = e$, where $V$ is a set of vertices and $E$ is composed of some $2$-element subsets of $V$ which represent edges. A \emph{realization} is some $p = (p_1, p_2, \ldots, p_v) \in \mathbb{R}^{vd}$, where each $p_i$ is the location of $v_i \in V$ in $\mathbb{R}^d$. This defines the framework $G(p)$. For some framework $G(p)$, the half edge-length squared function is $f_G(p): \mathbb{R}^{vd} \to \mathbb{R}^e$, where $f_G(p) = \frac{1}{2}(\ldots, |p_i - p_j|^2, \ldots)$ for $\{i, j\} \in E$.  Define a continuous \emph{flexing} of $G(p)$ as a differentiable one-parameter family of realizations including $p$ such that for any $q$ in the family, $f_G(q) = f_G(p)$.  A framework is \emph{locally rigid} if all its flexings are trivial (the Euclidean motions). A framework is \emph{locally flexible} if there exists a non-trivial flexing.

The problem of determining the local rigidity of a framework is very difficult.  To simplify the problem, we will restrict our focus to \emph{generic} realizations, defined as realizations whose coordinates are algebraically independent over the rationals.  For generic realizations, this problem becomes much easier and makes use of $df_G(p)$, which we will refer to as the \emph{rigidity matrix}.  In this $e \times vd$ matrix, each row represents an edge, and each column represents a coordinate of some vertex.  In the row representing the edge connecting $v_i$ and $v_j$, any given column will be $0$ if it does not represent $v_i$ or $v_j$.  If it represents the $k^{th}$ coordinate of $p_i$, the entry is the $k^{th}$ coordinate of $p_j$ minus the $k^{th}$ coordinate of $p_i$. 

Due to early results by Asimow and Roth \cite{ar}, we know that local rigidity is a generic property of the underlying graph, meaning that if it holds for one generic framework, it holds for all generic frameworks.  Thus, one can think of generic local rigidity as an inherent property of the graph.  The rank of the rigidity matrix is closely related to the local rigidity of a generic framework.  For graphs with at least $d+1$ vertices, we say a framework $G(p)$ is \emph{infinitesimally rigid} if $\rank df_G(p) = vd - \binom{d+1}{2}$.  We say it is \emph{infinitesimally flexible} if $\rank df_G(p) < vd - \binom{d+1}{2}$.
\begin{theorem} [Asimow and Roth \cite{ar}]\label{thm:rigidity}
A graph $G$ with at least $d+1$ vertices is generically locally rigid in $\mathbb{R}^d$ if and only if a generic realization is infinitesimally rigid.
\end{theorem}

Note that the rank of the rigidity matrix cannot be greater than $vd - \binom{d + 1}{2}$, because the Euclidean motions are always in the kernel of the rigidity matrix and there is a $\binom{d + 1}{2}$-dimensional space of them.  This provides an algorithm to check if a graph is \emph{generically locally rigid} (GLR) \cite{he, ght}:  given a graph, randomize its coordinates, generate the rigidity matrix modulo a large prime, and calculate its rank.  With no false positives and very few false negatives, this will decide the generic local rigidity of the graph. 

Next, for some graph $G$, let $K_v$ be the complete graph on the same set of vertices, that is, the graph such that $E$ consists of all $2$-element subsets of $V$.  We will define a framework $G(p)$ as \emph{globally rigid} when $f_G(p) = f_G(q)$ implies that $f_{K_v}(p) = f_{K_v}(q)$. This means that a framework is globally rigid when, for any other framework with the same edge lengths, all other pairwise distances are the same.  Clearly, all globally rigid frameworks are also locally rigid; however, not all locally rigid frameworks are globally rigid.  As an example, consider a generic realization $p$ of a quadrilateral in $\mathbb{R}^2$ with an edge along one diagonal [Figure~\ref{fig:quad}]. This framework is locally rigid, since there is no non-trivial continuous flexing. However, it is possible to reflect one part of the framework over the diagonal to produce another realization $q$ such that  $f_G(p) = f_G(q)$ but $f_{K_v}(p) \neq f_{K_v}(q)$.
\begin{figure}
\subfloat[]{\includegraphics[scale=1]{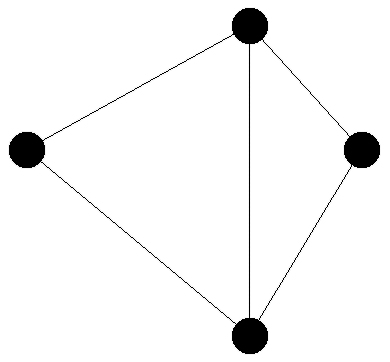}\label{fig:quad1}}\hspace{1in}
\subfloat[]{\includegraphics[scale=1]{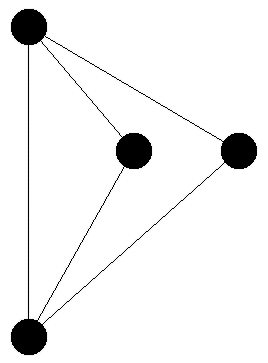}\label{fig:quad2}}
\caption{This framework is locally rigid in $\mathbb{R}^2$ but not
  globally rigid.}
\label{fig:quad}
\end{figure}

One can test for global rigidity using stresses.  A stress is some vector $\omega = (\ldots,~ \omega_{ij}, \ldots) \in \mathbb{R}^e$ for all $\{i, j\} \in E$.  An \emph{equilibrium stress} is a stress such that, for all vertices $v_i \in V$, \[ \sum_{j | \{i, j\} \in E} \omega_{ij} \cdot (p_j - p_i) = 0. \] From now on, by \emph{stress} we mean equilibrium stress, unless otherwise specified. 

\begin{proposition}\label{prop:rmstress}
The space of stresses of a framework $G(p)$ is precisely $\ker(df_G(p)^T)$.
\end {proposition}

\begin{proof}
Write the stress condition for each vertex $v_i$, and arrange them into a matrix such that every vector in the kernel is a stress.  It is not difficult to show that this matrix is exactly $df_G(p)^T$.
\end{proof}

A \emph{stress matrix} $\Omega$ is a $v \times v$ matrix satisfying the following conditions:
\begin{equation*} 
\Omega_{ij} = 
\begin{cases}
0 & \text{if } \{i,j\} \not\in E \text{ and } i \neq j\\ 
\omega_{ij} & \text{if } \{i,j\} \in E\\ 
- \sum_{j' \neq i}\Omega_{ij'}& \text{if } i = j 
\end{cases}
\end{equation*}

Each of the coordinate projections is in the kernel of $\Omega$, as is the vector $(1, \ldots, 1)$. This means the dimension of the kernel is at least $d+1$.  Also as a consequence, across each row, the vectors $p_1, p_2, \ldots, p_v$ fulfill an affine linear relation with the row's entries acting as coefficients.

\begin{theorem}[Connelly \cite{cn2}, Gortler-Healy-Thurston\cite{ght}]
\label{thm:strssmtrx}
A graph with at least $d+2$ vertices is generically globally rigid if and only if, for some generic realization, there is a stress matrix with nullity $d + 1$.
\end{theorem}

Connelly showed that this condition is sufficient; Gortler, Healy and Thurston showed that it is necessary as well, therefore implying that global rigidity is a generic property of the graph.  Furthermore, Gortler, Healy, and Thurston proposed a randomized algorithm to efficiently check if a graph is \emph{generically globally rigid} (GGR).  Given a graph, randomize its coordinates and create its rigidity matrix modulo a large prime.  Due to Proposition \ref{prop:rmstress}, we can find a random stress by selecting random vectors in $\ker(df_G(p)^T)$. Turn this stress into a stress matrix and check its rank; with no false positives and very few false negatives, this returns whether or not the graph is GGR.  We used this algorithm, as well as the algorithm described before, to experimentally check whether or not graphs were generically locally and globally rigid.

However, this is not an intuitive way of determining generic global rigidity, and a simpler process has eluded many mathematicians.   Some necessary conditions for generic global rigidity have been established by Hendrickson.  We define an edge of a framework as \emph{redundant} if one can remove it and be left with a locally rigid framework.  A framework is \emph{redundantly rigid} if all of its edges are redundant. We say that a graph is \emph{generically redundantly rigid} (GRR) if any of its generic frameworks are redundantly rigid.
\begin{theorem}[Hendrickson \cite{he}]\label{thm:Hendrickson}
If a graph in $\mathbb{R}^d$ has at least $d+2$ vertices and is generically globally rigid, then it is both generically redundantly rigid and vertex $(d+1)$-connected.
\end{theorem}
From now on, when we use the term $k$-connected, we mean vertex $k$-connected.  In $\mathbb{R}^1$, the two conditions of Theorem \ref{thm:Hendrickson} are equivalent to $2$-connectedness, and they are also sufficient for generic global rigidity.  In $\mathbb{R}^2$, due to results from Connelly \cite{cn2}, Jackson and Jord\'an \cite{jj}, we know that the conditions are sufficient as well.  Hendrickson conjectured that they are sufficient in all dimensions. However, Connelly \cite{cn} found the counterexample of $K_{5,5}$ in $\mathbb{R}^3$. He also generalized this into a class of complete bipartite graphs.

\begin{theorem}[Connelly \cite{cn}]\label{thm:bipartite}
Any complete bipartite graph $K_{a,b}$ in $\mathbb{R}^d$ such that $a + b = \binom{d+2}{2}$ and $a, b \geq d+2$ is $(d+1)$-connected and generically redundantly rigid, but not generically globally rigid.
\end{theorem}

We denote all graphs that violate Hendrickson's sufficiency conjecture and are not GGR as \emph{generically partially rigid} (GPR).  In addition to these complete bipartite graphs, the process of coning can also create GPR graphs.  \emph{Coning} a graph $G$ is the process of adding a vertex to $G$ and connecting it to every other vertex in $G$.
\begin{theorem}[Connelly and Whiteley \cite{cw}]
For any graph $G$, coning preserves the generic local, redundant, and global rigidity of $G$ from $\mathbb{R}^d$ to $\mathbb{R}^{d+1}$. It also transfers $(d+1)$-connectedness to $(d+2)$-connectedness.
\end{theorem}
\begin{corollary}
Coning a generically partially rigid graph in $\mathbb{R}^d$ creates a generically partially rigid graph in $\mathbb{R}^{d+1}$.
\end{corollary}
However, so far the only documented graphs that are GPR are complete bipartite graphs and their conings.  In a very recent paper \cite[8.3]{cn3}, Connelly posed some questions about the nature of GPR graphs in higher dimensions.  We will present two new classes of GPR graphs and answer two of Connelly's questions.  We find two GPR graphs in $\mathbb{R}^4$ and infinitely many in $\mathbb{R}^d$ for each $d \ge 5$.

In section $2$, we will present one of our main results for a class of graphs called $k$-chains and present some simple proofs, including proving when these graphs are GGR.  In section $3$, we determine under what conditions these graphs are GLR.  In section $4$, we do the same for GRR and prove the main result from section $2$.  In section $5$, we introduce a new graph construction and prove that it generates infinitely many GPR graphs.  Finally, we present some conjectures for further exploration in section $6$.
\end{section}


\begin{section}{Main Result for $k$-Chains}
For positive integers $a_1, a_2, \ldots, a_k $, the \emph{$k$-chain} $C_{a_1,a_2, \ldots, a_k}$ is the graph constructed as follows.  The vertex set $V$ is the union of $k$ disjoint sets of vertices $A_1$, $A_2$,$\ldots$, $A_k$ such that $|A_i| = a_i$.  For $1 \leq i \leq k-1$, there are edges between every vertex in $A_i$ and $A_{i+1}$, and the graph has no other edges.  Note that a 2-chain is simply a complete bipartite graph and that the 3-chain $C_{a_1, a_2, a_3}$ is the complete bipartite graph $K_{a_1+a_3, a_2}$. In particular, Connelly's GPR bipartite graphs can be characterized as $3$-chains. We are interested in characterizing when a $k$-chain is GPR.
\begin{theorem}
\label{main}
A $k$-chain $C_{a_1, a_2, \cdots, a_k}$ with $k \geq 4$ and $ \binom{d+2}{2}$ vertices is generically partially rigid if and only if it satisfies all of the following conditions:
\begin{enumerate}
\item $a_2, a_3, \ldots, a_{k-1} \geq d + 1$;
\item $a_2, a_{k-1} \geq d + 2$; and
\item there is no $i$ such that $a_i = a_{i+1} = d + 1$.
\end{enumerate}
\end{theorem}

The proof will occupy much of the rest of the paper. For $3$-chains with $v = \binom{d+2}{2}$, one must add the additional condition that $a_1 + a_3 \geq d + 2 $.  Note that this condition holds for any $k$-chain with $k \geq 4$ that fulfills the conditions of Theorem \ref{main}.

There are no $k$-chains satisfying the conditions of Theorem \ref{main} in $\mathbb{R}^3$. For $\mathbb{R}^4$, $v = \binom{6}{2} = 15$, so the only GPR examples in $\mathbb{R}^4$ are $C_{1, 6, 6, 2}$ and $C_{1, 6, 7, 1}$ [Figure~\ref{fig:c1662c1671}].
\begin{figure}
\subfloat[]{\includegraphics[scale=1]{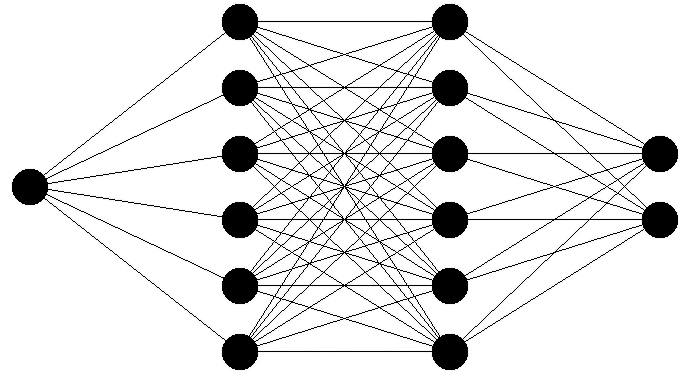}\label{fig:c1662}}\hspace{1in}
\subfloat[]{\includegraphics[scale=1]{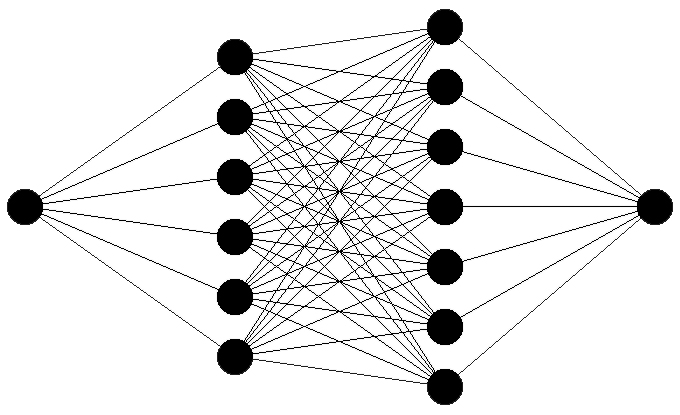}\label{fig:c1671}}
\caption{These are the only GPR $k$-chains in $\mathbb{R}^4$ with $\binom{d+2}{2}$ vertices.}
\label{fig:c1662c1671}
\end{figure}

\begin{proposition}
\label{dponeconnected}
A $k$-chain is $(d+1)$-connected if and only if it fulfills condition 1 of Theorem \ref{main}.
\end{proposition}
\begin{proof}
If $a_2, \ldots, a_{k-1} \geq d+1$, then removing any $d$ vertices leaves at least one vertex in each independent set, so the graph remains connected.  If not, then for some $i$, $2 \leq i \leq k-1$, $a_i \leq d$, so one can remove $A_i$, disconnecting the graph.
\end{proof}
\begin{proposition}
\label{globrigid}
Any $(d+1)$-connected $k$-chain with $k \geq 4$ and $\binom{d+2}{2}$ vertices is not generically globally rigid in $\mathbb{R}^d$.
\end{proposition} 
\begin{proof}
This $k$-chain is the subgraph of some complete bipartite graph.  Both independent sets of this complete bipartite graph have more than $d + 2$ vertices.  Since the complete bipartite graph has $\binom{d+2}{2}$ vertices, by Theorem \ref{thm:bipartite} it is GPR, and thus it is not GGR.  So, the $k$-chain is the subgraph of a graph which is not GGR, and so is not GGR itself.
\end{proof}
It remains to be determined when these graphs are GLR and when they are GRR.  
\end{section}


\begin{section}{Proof of Generic Local Rigidity}
In this section we show that $k$-chains that are $(d+1)$-connected are GLR.  We assume $(d+1)$-connectedness and $k \geq 4$ throughout this section. 

First note that $C_{a_1, a_2, \ldots, a_k}$ is a subgraph of the complete bipartite graph $K_{a_1+a_3+\cdots,a_2+a_4+\cdots}$, which has $\binom{d+2}{2}$ vertices with at least $d+2$ in each independent set. Due to Bolker and Roth \cite{br}, we can calculate the dimension of the space of stresses for a generic framework of this complete bipartite graph.

Let $A$ and $B$ be the independent sets of some complete bipartite graph.  Let $\Omega(A,B)$ be the space of stresses of a generic framework of the graph.  Additionally,  for some set of vectors $X = \{x_1, x_2, \ldots, x_k\}$, let $D(X)$ be the space of affine linear dependencies of $X$.  Finally, for a vector $v = (v_1, \ldots, v_n)$, let $\overline{v}$ be $(v_1, \ldots, v_n, 1)$.  Then let $D^2(X)$ be the set of linear dependencies of $\{\overline{x_1} \otimes \overline{x_1}, \overline{x_2} \otimes \overline{x_2}, \ldots, \overline{x_k} \otimes \overline{x_k}\}$, where $\otimes$ denotes the tensor product of two vectors.
\begin{theorem}[Bolker and Roth \cite{br}]
Given some complete bipartite graph $K_{A,B}$ such that $|A|, |B| \geq d+1$, let $C = A \cup B$.  Then for any generic realization, $\dim \Omega (A,B) = \dim D(A) \cdot \dim D(B) + \dim D^2(C)$.
\end{theorem}
This is actually a specific instance of Bolker and Roth's results.  Bolker and Roth provided a more general but more complicated formula for all frameworks, but we are only interested in generic frameworks.
\begin{remark}\label{remark:br}
For a generic set of points $X$,
\begin{align*} 
\dim D(X) &=
\begin{cases}
0 & \text{if } |X| \leq d+1\\ 
|X| - d - 1 & \text{if } |X| > d+1\\ 
\end{cases}  
\\
\dim D^2(X) &=
\begin{cases}
0 & \text{if } |X| \leq \binom{d+2}{2}\\ 
|X| - \binom{d+2}{2} & \text{if } |X| >\binom{d+2}{2}\\ 
\end{cases}
\end{align*}
\end{remark}
\begin{corollary}\label{cor:bipstress}
Suppose $v \leq \binom{d+2}{2}$. For any generic realization of $K_{A,B}$ with $|A|,|B| \geq d+1$, $\dim \Omega (A,B) = (|A| - d - 1) \cdot (|B| - d - 1)$.  If $|A| < d+1$ or $|B| < d+1$, then $\dim \Omega (A,B) = 0$.
\end{corollary}
From the corollary, it is possible to compute the dimension of the space of stresses for the bipartite graph $K_{a_1+a_3+\cdots,a_2+a_4+\cdots}$. If we let $k = a_1+a_3+\cdots$ and $l = a_2+a_4+\cdots$, then the dimension is $(k - d - 1)(l - d - 1)$, and thus the rank of the rigidity matrix is $$kl - (k - d - 1)(l - d - 1) = (k+l)(d+1) - (d+1)^2 = (k+l)d - \binom{d+1}{2}$$ since $k+l = \binom{d+2}{2}$. Thus this complete bipartite graph is GLR, as also indicated by Theorem \ref{thm:bipartite}.

When we remove some edges from a GLR graph, it is possible to determine whether the new graph is GLR by examining the space of stresses.
\begin{proposition}\label{prop:redrigid}
Let G(p) be a generic, locally rigid framework, and let $e_1, \ldots, e_n$ be some edges of G.  Then $G\setminus \{e_1, \ldots, e_n\}$ is generically locally rigid if and only if, for any $a_1, \ldots, a_n \in \mathbb{R}$, there exists a stress on G(p) with values $a_1, \ldots, a_n$ on $e_1, \ldots, e_n$.
\end{proposition}
\begin{proof}
We use induction on $n$.  

[Base Case $\Rightarrow$] For $n = 1$, first suppose $G\setminus\{e_1\}$ is GLR.  Since $G(p)$ is locally rigid, $\rank df_G(p) = \rank df_{G\setminus\{e_1\}}(p)$.  Adding $e_1$ to $G\setminus\{e_1\}$ does not increase the rank of the rigidity matrix, so it increases the dimension of $\ker(df_G(p)^T)$ by $1$.  This means adding $e_1$ adds a new dimension of stresses, which is only possible if there is some stress with a non-zero value on $e_1$.  By scaling this stress, we can achieve any prescribed value on $e_1$.

[Base Case $\Leftarrow$] Assume there is some stress with a non-zero value on $e_1$.  Removing one edge can decrease the dimension of $\ker (df_G(p)^T)$ by at most $1$.  Moreover, there is a stress with a non-zero value on $e_1$, and since this stress cannot exist without $e_1$, removing $e_1$ must decrease $\dim\ker(df_G(p)^T)$ by exactly $1$.  But, the number of rows of $df_G(p)$ also decreases by $1$, so $\rank df_{G}(p)$ stays the same.  Therefore, $\rank df_G(p) = \rank df_{G\setminus\{e_1\}}(p)$, and so $G\setminus\{e_1\}$ is GLR.

[Inductive Step $\Rightarrow$] Assume that for some $n$, when $G\setminus \{e_1, \ldots, e_n\}$ is GLR, there is a stress on $G(p)$ with any $a_1, \ldots, a_n$ on $e_1, \ldots, e_n$.  Then assume that $G\setminus \{e_1, \ldots, e_{n+1}\}$ is GLR.  Let $H = G\setminus \{e_1, \ldots, e_n\}$.  First, note that $H$ is also GLR, so we can create a stress on a generic framework $G(p)$ with values $a_1, \ldots, a_n$ on $e_1, \ldots, e_n$.  Call the stress we create $\omega$.  Because $H\setminus \{e_{n+1}\}$ is GLR, we can create some stress of $H(p)$ with any value we like on $e_{n+1}$ by Base Case $\Rightarrow$, and we can artificially extend it to a stress of $G(p)$ with values of $0$ on $e_1, \ldots, e_n$.  We give this stress the value on $e_{n+1}$ such that, when we compose it with $\omega$, we create a stress with values $a_1, \ldots, a_{n+1}$ on $e_1, \ldots, e_{n+1}$.

[Inductive Step $\Leftarrow$] Assume that for some $n$, if we can find a stress with any value on $e_1, \ldots, e_n$, $G\setminus \{e_1, \ldots, e_n\}$ is GLR.  Then, suppose we can find some stress with any value we want on $G\setminus \{e_1, \ldots, e_{n+1}\}$.  By the inductive hypothesis, $H$ is GLR.  If we set $a_1, \ldots, a_{n}$ to all be zero, then we can find a stress on $H$ with any value we wish on $e_{n+1}$.  So, by Base Case $\Leftarrow$, $G\setminus \{e_1, \ldots, e_{n+1}\}$ is GLR.
 \end{proof}
 
 \begin{corollary}\label{cor:redrigid}
 A graph $G$ is generically redundantly rigid in $\mathbb{R}^d$ if and only if it is generically locally rigid in $\mathbb{R}^d$ and there is a non-zero stress on every edge of $G$.
 \end{corollary}
 \begin{proof}
 If there is a non-zero stress on every edge of $G$, then by scaling, we can find a stress of any value we want on any edge of $G$.  Thus, by Proposition \ref{prop:redrigid}, each edge is redundant and $G$ is GRR.  
 
 On the other hand, if the only stress on some edge is the zero stress, we cannot find a stress of any value on that edge.  Thus, by Proposition \ref{prop:redrigid}, $G$ is not GRR.
 \end{proof}
 
Now we need to show that a $(d+1)$-connected $k$-chain is GLR. Recall that the $k$-chain is a subgraph of a complete bipartite graph, and by Theorem \ref{thm:bipartite}, that complete bipartite graph is GLR. Therefore, it is sufficient to demonstrate that the edges removed from the complete bipartite graph can take stresses of any value. Pick any two vertices which are not connected in the $k$-chain, but are connected in the complete bipartite graph.  We will show that there exists some stress with a non-zero value on the edge between these two vertices and values of zero on all other removed edges.

Suppose the two vertices come from the sets $A_i$ and $A_j$, assuming without a loss of generality that $i < j$. Note that $i - j$ is odd, since the removed edges must come from different independent sets of the complete bipartite graph. Furthermore, it is also evident that $i - j \geq 3$. Pick $d+1$ vertices from each of $A_{i+1}, \ldots, A_{j-1}$. Use these vertices and the two vertices in $A_i$ and $A_j$ to form $C_{1, d+1, \ldots, d+1, 1}$, denoted by $\Upsilon$. We will show that for generic realizations, $\Upsilon$ has a zero-dimensional space of stresses and that the graph obtained by connecting the two vertices at the ends, denoted by $\Upsilon '$, has a $1$-dimensional space of stresses.

Reorder the independent sets of $\Upsilon$ by $A_i, A_{i+2}, \ldots, A_{j-1},  A_{i+1}, A_{i+3}, \ldots, A_{j}$.  Because $\Upsilon$ is a bipartite graph, there are no edges between any two vertices in $A_i \cup A_{i+2} \cup \ldots \cup A_{j-1}$; the same can be said of the vertices in $A_{i+1} \cup A_{i+3} \cup \ldots \cup A_j$.  Therefore,  the upper-left and bottom-right corners of the stress matrix of $\Upsilon$ have values of zero on the non-diagonal entries. Moreover, Bolker and Roth \cite{br} demonstrated that the stress matrix has values of zero on the diagonal entries.  Furthermore, the stress matrix is symmetric across the diagonal, and because each row fulfills an affine linear relation with the projection vectors, so do the columns.  Therefore, it is sufficient to examine the upper-right corner of the matrix, keeping in mind the affine linear relations on both the rows and the columns.  We will use the following remark to analyze the stress matrix.

\begin{remark}\label{affine}
If $d+1$ generic vectors $v_1, v_2, \ldots, v_{d+1} \in \mathbb{R}^d$ satisfy an affine linear relation, that is, for $a_1, a_2, \ldots, a_{d+1} \in \mathbb{R}$
\begin{equation*}
a_1v_1+a_2v_2+\ldots+a_{d+1}v_{d+1} = 0
\end{equation*}
\begin{equation*}
a_1+a_2+\ldots+a_{d+1} = 0
\end{equation*}
Then $a_1 = a_2 = \cdots = a_{d+1} = 0$. This can easily be seen by solving the second equation for $a_{d+1}$ and substituting into the first equation. Then we get a linear relation on $d$ generic vectors in $\mathbb{R}^d$, which forces each of the coefficients to be $0$.
\end{remark}

The upper-right corner of the stress matrix has the following shape.
\begin{equation*}
\begin {tabular} {r|cccccc|}
& $A_{i+1}$ & $A_{i+3}$ & $A_{i+5}$ & $\cdots$ & $A_{j-2}$ & $A_{j}$ \\
\hline
$A_i$ & $*_{1}$ & 0 & $\cdots$ & $\cdots$ & $\cdots$ & 0\\
$A_{i+2}$ & $*_{2}$ & * & 0 & $\cdots$ & $\cdots$ & 0\\
$A_{i+4}$ & 0 & * & * & 0 & $\cdots$ & 0\\
$\vdots$ & $\vdots$ & $\ddots$ & $\ddots$ & $\ddots$ & $\ddots$ & $\vdots$\\
$A_{j-3}$ & 0 & $\cdots$  &  0 & * & * & 0\\
$A_{j-1}$ & 0 & $\cdots$  & $\cdots$ & 0 & * & *\\
\hline
\end {tabular}
\end{equation*}

The asterisks represent all possible sets of non-zero entries, corresponding to the edges of $\Upsilon$. All of the edges are between vertices in $A_n$ and $A_{n+1}$ for some $n$, causing the asterisks to form a ``staircase" pattern. Consider $*_{1}$, a $1$ by $d+1$ block of entries. These $d+1$ entries fulfill an affine linear relation among generic vectors across the first row. By Remark \ref{affine}, every entry in $*_1$ is therefore $0$. Next, consider $*_2$, a $d+1$ by $d+1$ block of entries.  Looking at the first $d+1$ columns of the upper right corner of the stress matrix, $*_{2}$ must be uniformly $0$ as well because the projection vectors fulfill an affine linear relation on each column. Working down the ``staircase'' by alternately solving for rows and columns, each of the asterisks must be uniformly $0$. Hence, the only stress is the zero stress. 

Now consider $\Upsilon '$. Since $i - j$ is odd, let $i - j + 1 = 2l$.  The graph $\Upsilon '$ contains $(2l-2)(d+1) + 2 = 2(l-1)(d+1) + 2$ vertices and $(2l - 3)(d+1)^2 + 2(d+1) + 1$ edges. $\Upsilon '$ is a subgraph of some complete bipartite graph with the same vertices. Each of the independent sets of this complete bipartite graph has $(l - 1)(d + 1) + 1$ vertices. The complete bipartite graph has $[(l-1)(d+1) + 1]^2$ edges, and so by Corollary \ref{cor:bipstress}, it has a $[(l-1)(d+1) + 1 - d - 1]^2 = [(l-2)^2(d+1)^2 + 2(l-2)(d+1) + 1]$-dimensional space of stresses.  $\Upsilon '$ results from the removal of $[(l-1)(d+1) + 1]^2 - [(2l - 3)(d+1)^2 + 2(d+1) + 1] = (l-2)^2(d+1)^2 + 2(l-2)(d+1)$ edges from the complete bipartite graph. Each edge removed reduces the rank of the rigidity matrix by at most $1$, and so reduces the dimension of the space of stresses by at most $1$. Therefore, after removing $(l-2)^2(d+1)^2 + 2(l-2)(d+1)$ edges, the dimension of the space of stresses is at least $1$.

Finally, since $\Upsilon$ has a zero-dimensional space of stresses and $\Upsilon '$ has a positive-dimensional space of stresses, there must be a non-zero stress on the edge connecting the two vertices. In fact, $\Upsilon'$ has exactly a $1$-dimensional space of stresses, since removing one edge forces the space of stresses to be zero-dimensional.  This implies that each of the removed edges can be written as a linear combination of the remaining edges in the rigidity matrix, meaning that each of these edges is responsible for a single independent dimension of stresses.

By composing the stresses of the subgraphs found above, we can obtain any value we want on the removed edges of the complete bipartite graph. This leads to the following result:
\begin{lemma}\label{lemma:section3} Any $(d+1)$-connected $k$-chain $C_{a_1, a_2, \cdots, a_k}$ with $k \geq 4$ and $ \binom{d+2}{2}$ vertices is generically locally rigid in $\mathbb{R}^d$.
\end{lemma}
\begin{proof}
$C_{a_1, a_2, \cdots, a_k}$ is a subgraph of a complete bipartite graph with the same vertices, which has already been proved to be GLR.  Moreover, as shown above, we can put arbitrary stresses on all of the edges that must be removed to create $C_{a_1, a_2, \cdots, a_k}$. By Proposition \ref{prop:redrigid}, the $k$-chain is GLR.
\end{proof}
\end{section}


\begin{section}{Proof of Generic Redundant Rigidity}
Now that we know the $k$-chains in question are GLR if condition 1 of Theorem \ref{main} is satisfied (which we will assume throughout the section), it remains to be determined under what conditions they are GRR.  According to Corollary \ref{cor:redrigid}, a framework is redundantly rigid if and only if there is some stress with non-zero entries on every edge.  Consequently, we will find the space of stresses of the $k$-chains.

By Proposition \ref{prop:rmstress}, the space of stresses is the kernel of the transpose of the rigidity matrix.  Because the graph is GLR, for any generic realization $p$, the dimension of the space of stresses is $e - \rank df_G(p) = e - vd + \binom{d+1}{2}$, where $v = \binom{d+2}{2}$.

Now we consider all the $3$-chains $C_{a_i,a_{i+1},a_{i+2}}$ that are subgraphs of our $k$-chain.  Call this set of $3$-chains the \emph{$3$-chain cover} of the $k$-chain.  Note that any stress of one of these $3$-chains is also a stress of the entire $k$-chain.  Moreover, we present the following lemma.

\begin{lemma}\label{lemma:covering}  
Let $C_{a_1, a_2, \cdots, a_k}$ be a $(d+1)$-connected $k$-chain with $k \geq 4$ and $ \binom{d+2}{2}$ vertices.  Then the space of stresses of $C_{a_1, a_2, \cdots, a_k}$ is precisely the space of stresses of the $3$-chain cover of $C_{a_1, a_2, \cdots, a_k}$.
\end{lemma}
\begin{proof} To find the dimension of the space of stresses of the $3$-chain cover, use the inclusion-exclusion principle.  The overlap among the stresses stems from the $2$-chains shared by adjacent $3$-chains.  So, using Corollary \ref{cor:bipstress} and some simple algebra, the dimension of the space of stresses of the $3$-chain cover is:
\begin{align*}
&\sum_{i=2}^{k-1} (a_{i-1} + a_{i+1} - d - 1)(a_i - d - 1) - \sum_{i=2}^{k-2}(a_i-d-1)(a_{i+1}-d-1)\\
&= \sum_{i=1}^{k-1} a_ia_{i+1} - \left(\sum_{i=1}^{k} a_i\right)(d+1) + (d+1)^2\\
&= e - v(d+1) + (d+1)^2
 \end{align*}
 
If $v = \binom{d+2}{2}$, the reader can verify that this is also $e - vd + \binom{d+1}{2}$.  Since the stresses of the $3$-chain cover constitute a subspace of the total space of stresses with equal dimension, they account for the entire space of stresses of $C_{a_1, a_2, \cdots, a_k}$.
\end{proof}
 
Note that if a $3$-chain has a positive-dimensional space of stresses, we can find some stress with non-zero values on every entry. To see this, first note that a $3$-chain is a complete bipartite graph.  If the graph has a positive-dimensional space of stresses in $\mathbb{R}^d$, then some edge has a non-zero stress on it.  However, complete bipartite graphs are completely symmetric across their edges with respect to the existence of non-zero stresses.  If there is a stress with entries of $0$ on some edge, we can use the symmetry of the graph to find stresses with a non-zero value on that edge and then add the stresses.  Thus, it is possible to find a stress with non-zero values on every edge.
 
Now we are equipped with all the tools necessary to examine redundant rigidity. This leads to the following lemma.
\begin{lemma}\label{lemma:section4}
A $(d+1)$-connected $k$-chain $C_{a_1, a_2, \cdots, a_k}$ with $k \geq 4$ and $ \binom{d+2}{2}$ vertices is generically redundantly rigid if and only if 
\begin{enumerate}
\item $a_2, a_{k-1} \geq d + 2$ and 
\item there is no $i$ such that $a_i = a_{i+1} = d + 1$.
\end{enumerate}
\end{lemma}
\begin{proof}
First, note that by Lemma \ref{lemma:section3}, $C_{a_1, a_2, \cdots, a_k}$ is GLR.  We will apply Corollary \ref{cor:redrigid} directly in the rest of this proof, so we only have to determine whether the stresses of the graphs are non-zero on every edge.

[$\Rightarrow$] Suppose either condition $1$ or condition $2$ does not hold. If $a_2 < d+2$, then by Corollary \ref{cor:bipstress}, the $3$-chain $C_{a_1,a_2,a_3}$ (or bipartite graph $K_{a_1+a_3,a_2}$) will have a zero-dimensional space of stresses, and this is the only $3$-chain which includes the edges connecting $A_1$ and $A_2$.  Any stress on the $k$-chain will have entries of 0 on these edges, meaning that they are not redundant and as a consequence, the graph is not GRR.  The same argument applies to $a_{k-1}$.  

Moreover, suppose that for some $i$, $a_i = a_{i+1} = d + 1$.  Then by Corollary \ref{cor:bipstress}, both $C_{a_{i-1}, a_i, a_{i + 1}}$ and $C_{a_i, a_{i+1}, a_{i+2}}$ have only the zero stress.  These are the only two $3$-chains that cover the edges between $A_i$ and $A_{i+1}$, so by Lemma \ref{lemma:covering}, any stress on the $k$-chain will have entries of $0$ on these edges.

[$\Leftarrow$] Assume that $a_2, a_{k-1} \geq d + 2$ and there is no $i$ such that $a_i = a_{i+1} = d+1$. Firstly, there is a non-zero stress covering the edges between $A_1$ and $A_2$. To see this, consider $C_{a_1,a_2,a_3}$, where each of $a_1+a_3$ and $a_2$ is at least $d+2$, so this $3$-chain or bipartite graph has a non-zero space of stresses. Hence, each edge between $A_1$ and $A_2$ has a non-zero stress covering it. The same argument can be applied to the edges between $A_{k-1}$ and $A_k$.

For the other edges, there are two cases to consider.  In the first case, for all $3 \leq i \leq k-2$, $a_i \geq d+2$.  In this case, there is obviously a stress with non-zero values everywhere.  Otherwise, there exists some $3 \leq i \leq k-2$ such that $a_i = d+1$.  Then we know that $a_{i-1}, a_{i+1} \geq d+2$, so using Corollary \ref{cor:bipstress} on $C_{a_{i-2},a_{i-1},a_i}$ and $C_{a_i,a_{i+1},a_{i+2}}$, we know that the edges between $A_{i-1}$ and $A_i$ and the edges between $A_i$ and $A_{i+1}$ have non-zero stresses covering them.  Thus, the $k$-chain is GRR.  
\end{proof}

We are now able to prove Theorem \ref{main}.

\begin{proof}

[$\Rightarrow$] Suppose that the conditions do not all hold. If condition $1$ fails, then by Proposition \ref{dponeconnected}, the $k$-chain is not $(d+1)$-connected, and therefore not GPR. If either condition $2$ or condition $3$ fails, then by Lemma \ref{lemma:section4}, the $k$-chain is not GRR.

[$\Leftarrow$]  Suppose that the conditions all hold. Since condition $1$ holds, by Proposition \ref{dponeconnected}, the graph is $(d+1)$-connected, and by Lemma \ref{lemma:section3}, it is GLR. Since conditions $2$ and $3$ hold, by Lemma \ref{lemma:section4}, the $k$-chain is GRR. Finally, by Proposition \ref{globrigid}, the graph is not GGR. Therefore, the $k$-chain is GPR.
\end{proof}

\end{section}


\begin{section}{Graph Attachments in $\mathbb{R}^5$}
Theorem \ref{main} completely characterizes GPR $k$-chains with $\binom{d+2}{2}$ vertices. However, we also found a new class of GPR graphs which are not necessarily bipartite. Here we present a specific case, which we expect can be generalized in the future. Consider in $\mathbb{R}^5$ the $4$-chain $C_{2,3,5,4}$, and another arbitrary graph $G = (V, E)$ with at least $6$ vertices. We \emph{attach} $C_{2,3,5,4}$ to $G$ by letting $A_1$ and $A_4$ be disjoint $2$-element and $4$-element subsets of vertices in $V$, with none of the vertices of $A_2$ and $A_3$ in $V$. The set of edges precisely consists of all the edges in $G$ and $C_{2,3,5,4}$. Name the resulting graph $G ~ {\cup}_{2,4} ~ C_{2,3,5,4}$.


\begin{theorem}\label{thm:main2}
Let $G$ be a generically redundantly rigid and $6$-connected graph in $\mathbb{R}^5$ with at least $6$ vertices. Then $G ~ {\cup}_{2,4} ~ C_{2,3,5,4}$ is generically partially rigid.
\end{theorem}

First we need to show that the new graph is GLR. To do this, we carefully examine the graph $K_6 ~ {\cup}_{2,4} ~ C_{2,3,5,4}$ [Figure \ref{fig:ring}\subref{fig:ring1}].

\begin{figure}
\subfloat[]{\includegraphics[scale=1]{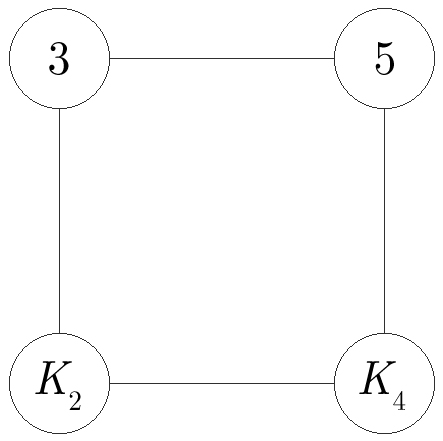}\label{fig:ring1}}\hspace{1in}
\subfloat[]{\includegraphics[scale=1]{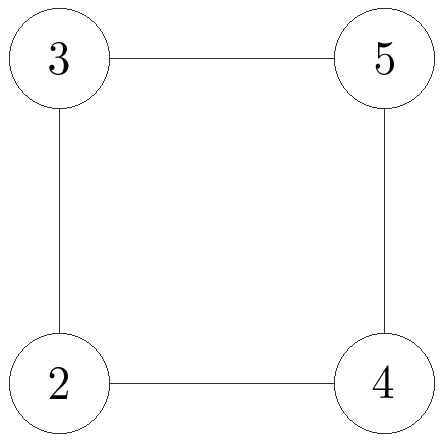}\label{fig:ring2}}
\caption{Each node represents a set of vertices: the numbers represent independent sets of the size indicated, and $K_2$ and $K_4$ represent complete graphs.  The lines represent edges between every combination of vertices in the nodes connected.  On left: $K_6 ~ {\cup}_{2,4} ~ C_{2,3,5,4}$.  On right: Graph obtained from left by deleting edges of $K_2$ and $K_4$.}
\label{fig:ring}
\end{figure}

\begin {proposition}\label{prop:rigidunion}
The graph $K_6 ~ {\cup}_{2,4} ~ C_{2,3,5,4}$ is generically locally rigid in $\mathbb{R}^5$.
\end {proposition}
\begin {proof}
We have not yet found a conceptual proof of this fact. However, using the algorithm for testing generic local rigidity described earlier, we have found one locally rigid realization, and the algorithm cannot return a false positive for a graph being GLR, since the rank of the rigidity matrix can only decrease due to non-generic realizations and special primes. This proves that the graph is GLR.
\end {proof}

We want to know which edges of $K_6 ~ {\cup}_{2,4} ~ C_{2,3,5,4}$ are redundant. It is possible to do so by finding its stresses. The graph $K_6 ~ {\cup}_{2,4} ~ C_{2,3,5,4}$ has $56$ edges, and its rigidity matrix has rank $55$. Hence, it has a $1$-dimensional space of stresses. We can easily identify all the stresses of the graph.

Remove the edges of the $K_2$ and $K_4$ subgraphs [Figure \ref{fig:ring}\subref{fig:ring2}]. The remaining graph is the bipartite graph $K_{7,7}$.   By Corollary \ref{cor:bipstress}, $K_{7,7}$ has a $1$-dimensional space of stresses, which is a subspace of the $1$-dimensional space of stresses of $K_6 ~ {\cup}_{2,4} ~ C_{2,3,5,4}$. Hence, every stress in $K_6 ~ {\cup}_{2,4} ~ C_{2,3,5,4}$ must also be a stress of $K_{7,7}$.  Moreover, by symmetry, if there is a non-zero stress on one edge of $K_{7,7}$, there is a non-zero stress on every edge.  Therefore, all the edges of $K_6 ~ {\cup}_{2,4} ~ C_{2,3,5,4}$ have non-zero stresses except for the edges on $K_2$ and $K_4$. In particular, by Corollary \ref{cor:redrigid}, each of the edges of the subgraph $C_{2,3,5,4}$ are redundant. More generally, for any $K_i ~ {\cup}_{2,4} ~ C_{2,3,5,4}$ with $i \geq 6$, we can find a non-zero stress on the edges of $C_{2,3,5,4}$, making each of these edges redundant.



We will also use a very useful technique called a Hennenberg operation to examine the rigidity of our graphs.  One constructs a new graph with the Hennenberg operation as follows: begin with a graph $G$ and some dimension $d$, and pick any two vertices $v_i$ and $v_j$ with an edge between them. Remove this edge, and add a vertex $v'$ to $G$, connecting it to $v_i$, $v_j$, and $d-1$ other vertices.  This new graph, denoted by $G'$, is obtained from $G$ by a \emph{Hennenberg operation}.

Connelly described the following theorem regarding Hennenberg operations.

\begin{theorem}[Connelly \cite{cn2}]\label{thm:Hennenberg}
If $G$ is generically locally rigid in $\mathbb{R}^d$, and $G'$ is obtained from $G$ by an Hennenberg operation, then $G'$ is generically locally rigid in $\mathbb{R}^d$.
\end{theorem}



\begin{lemma}\label{lemma:Knrigid}
The graph $K_n ~ {\cup}_{2,4} ~ C_{2,3,5,4}$ is generically locally rigid in $\mathbb{R}^5$ for all $n \geq 6$.
\end{lemma}



\begin{proof} We will prove the lemma by induction.  For the base case, by Proposition \ref{prop:rigidunion}, $K_6 ~ {\cup}_{2,4} ~ C_{2,3,5,4}$ is generically locally rigid in $\mathbb{R}^5$.

For the inductive step, assume $K_n ~ {\cup}_{2,4} ~ C_{2,3,5,4}$ is generically locally rigid in $\mathbb{R}^5$.  Use a Hennenberg operation to create a new graph, choosing $v_i$ and $v_j$ to be any vertices in $K_n$ and connecting $v'$ to $4$ other vertices in $K_n$.  This new graph is GLR in $\mathbb{R}^5$ by the above theorem. Finally, we can add the edge between $v_i$ and $v_j$, and connect $v'$ to the rest of the vertices in $K_n$.   This constructs the graph $K_{n+1} ~ {\cup}_{2,4} ~ C_{2,3,5,4}$, which is GLR because it is constructed by adding edges to a GLR graph.
\end{proof}

The following important result arises from the previous lemma.

\begin {lemma}
\label{lemma:attachredrigid}
Let $G$ be a generically redundantly rigid graph in $\mathbb{R}^5$ with at least $6$ vertices. Then $G ~ {\cup}_{2,4} ~ C_{2,3,5,4}$ is generically redundantly rigid in $\mathbb{R}^5$.
\end {lemma}
\begin {proof}
There are two cases to be examined. We can either remove an edge from $G$ or from $C_{2,3,5,4}$. We want to show that each of the resulting graphs is GLR.

In the first case, we remove an edge from $G$ to get $G'$, which is still GLR. Suppose $G'$ has $v$ vertices.  $G'$ is a subgraph of $K_v$, so by Proposition \ref{prop:redrigid} we can assign stresses for $K_v$ with any values on the edges of $K_v \setminus G'$. We know that $K_v ~ {\cup}_{2,4} ~ C_{2,3,5,4}$ is GLR. By assigning zero stresses to the edges of $C_{2,3,5,4}$, it is possible to assign stresses for $K_v ~ {\cup}_{2,4} ~ C_{2,3,5,4}$ with any values on $K_v  \setminus G' $. Thus, by Proposition \ref{prop:redrigid}, $G' ~ {\cup}_{2,4} ~ C_{2,3,5,4}$ is GLR.

In the second case, we remove an edge from $C_{2,3,5,4}$ and denote the resulting graph $C'_{2,3,5,4}$.  Now suppose that $G$ has $v$ vertices.  The graph $K_v ~ {\cup}_{2,4} ~ C_{2,3,5,4}$ is GLR, and each of the edges of the subgraph $C_{2,3,5,4}$ is redundant. Therefore, by Proposition \ref{prop:redrigid}, $K_v ~ {\cup}_{2,4} ~ C'_{2,3,5,4}$ is GLR.  Finally, using the same argument as before, it is possible to create stresses for $K_v$ with any values on $K_v \setminus G$.  These stresses will still exist on $K_v ~ {\cup}_{2,4} ~ C'_{2,3,5,4}$, so by Proposition \ref{prop:redrigid}, $G ~ {\cup}_{2,4} ~ C'_{2,3,5,4}$ is GLR.  Consequently, $G ~ {\cup}_{2,4} ~ C_{2,3,5,4}$ is GRR.
\end {proof}

\begin {remark}
\label{remark:attachconnected}
Let $G$ be any $6$-connected graph. Then $G ~ {\cup}_{2,4} ~ C_{2,3,5,4}$ is still $6$-connected. Removing $5$ vertices from $G$ will not disconnect $G ~ {\cup}_{2,4} ~ C_{2,3,5,4}$. It takes the removal of $7$ vertices to isolate $A_2$, $7$ vertices to isolate $A_3$, and $6$ vertices to isolate $C_{3,5}$.
\end {remark}

To conclude our argument, we will use the following construction.  Consider a graph $G$ with a subgraph $H$, and suppose $H$ has $v$ vertices.  Consider another graph $H'$ with at least $v$ vertices.  We \emph{replace} $H$ with $H'$ as follows. Begin with $G$.  Replace the vertices of $H$ with the vertices of $H'$. Create an injective mapping $\iota: H \to H'$. For each edge connecting vertex $g \in G$ to vertex $h \in H$, add an edge connecting $g$ to $\iota(h)$.  Finally, add all of the edges in $H'$.  The new graph is the \emph{replacement} of $H$ with $H'$.  Intuitively, replacing $H$ with $H'$ consists of removing $H$ and placing $H'$ in its place.

\begin{remark}
Given $K_6 ~ {\cup}_{2,4} ~ C_{2,3,5,4}$ and some $G$ with at least $6$ vertices, replacing $K_6$ with $G$ is equivalent to creating $G ~ {\cup}_{2,4} ~ C_{2,3,5,4}$.
\end{remark}

\begin{lemma}\label{notGGR}
Suppose $G$ is not generically globally rigid in $\mathbb{R}^d$ but contains the subgraph $H$ which is generically globally rigid.  Let $H'$ be a graph with at least as many vertices as $H$.  Then replacing $H$ with $H'$ results in a graph that is not generically globally rigid in $\mathbb{R}^d$.
\end{lemma}

\begin{proof}
Since $G$ is not GGR, we can find a generic framework $G(p)$ and a non-equivalent framework $G(q)$, up to Euclidean motions, in $\mathbb{R}^d$. Let $p_1, \ldots, p_v,\allowbreak q_1, \ldots, q_v$ represent the locations of vertices in $H$, and $p_{v+1}, \ldots,\allowbreak p_w,\allowbreak q_{v+1}, \ldots, q_w$ represent the locations of vertices in $G \setminus H$.

It is possible to transform $G(q)$ into an equivalent framework $G(q')$ with $p_i = q'_i$ for $i = 1, \ldots, v$. First, through translations, make $q'_1 = p_1$. Since $H$ is GGR, any other realization of vertices in $H$ must be equivalent up to Euclidean motions. Hence, one can reflect and rotate the entire framework to ensure $q'_i = p_i$ for $i = 2, \ldots, n$. Finally, since $G(p)$ and $G(q)$ are non-equivalent frameworks, $p_i$ and $q'_i$ are not all the same for $i = v+1, \ldots, w$. We have shown that for any generic framework $G(p)$, there exists a non-equivalent framework with the location of the vertices in $H$ the same.

Now connect all the edges of $H$ for both frameworks, which has the same effect as replacing $H$ with $K_v$ in both frameworks. Name this new graph $G'$. Since no edges are added between $G \setminus H$ and $H$, both $p$ and $q'$ preserve the edge lengths of $G'$. Any generic realization of $G'$ is also a generic realization of $G$. Therefore, for any generic framework $G'(p)$, there is another non-equivalent framework $G'(q')$, implying that $G'$ is not GGR.

Next, replacing $H$ with any $K_i$ for $i \geq v$ results in a graph that is not GGR. The case $i = v$ is already proved. The graph obtained by replacing $H$ with $K_i$, which we shall denote as $G''$, contains $G'$ as a subgraph. Starting with $G'$, $p$ and $q'$ as previously described, do the following to both frameworks: add $i-v$ vertices, connect them to all the vertices in $K_v$ only, project them onto the same set of locations in $\mathbb{R}^d$ for both realizations, and finally, ensure that $p$ is still generic. The new graph formed is precisely $G''$. No edges are added between any of the points in $K_i$ and $G'' \setminus K_i$, so $G''(q')$ is a non-equivalent realization of $G''(p)$. Finally, all generic realizations of $G''$ must have the points in the subgraph $G'$ be generic as well, so a non-equivalent realization can be found for any generic realization $G''(p)$. In this way, $G''$ is not GGR for any $i \geq v$.

Finally, the replacement of $H$ with $H'$ is a subgraph of the graph obtained by replacing $H$ with $K_i$ for some $i$, so replacing $H$ with $H'$ results in a graph that is not generically globally rigid.
\end{proof}

\begin{corollary}
\label{cor:attachglobrigid}
Let $G$ be any graph in $\mathbb{R}^5$ with at least $6$ vertices. Then $G ~ {\cup}_{2,4} ~ C_{2,3,5,4}$ is not generically globally rigid.
\end{corollary}

\begin{proof}
We examine $K_6 ~ {\cup}_{2,4} ~ C_{2,3,5,4}$. As discussed before, the only stresses of this graph come from the subgraph $K_{7,7}$. Consequently, $K_6 ~ {\cup}_{2,4} ~ C_{2,3,5,4}$ has the same stress matrix, up to scale, as $K_{7,7}$ for equivalent realizations. Both stress matrices have the same nullity and as a consequence of Theorem \ref{thm:strssmtrx}, $K_6 ~ {\cup}_{2,4} ~ C_{2,3,5,4}$ is not GGR.

On the other hand, $K_6 ~ {\cup}_{2,4} ~ C_{2,3,5,4}$ contains a subgraph $K_6$ which is GGR. Replacing $K_6$ with any $G$ with at least $6$ vertices forms the graph $G ~ {\cup}_{2,4} ~ C_{2,3,5,4}$, and by Lemma \ref{notGGR}, this graph is not GGR.
\end{proof}


This completes the proof of Theorem \ref{thm:main2}. Specifically, Lemma \ref{lemma:attachredrigid} shows generic redundant rigidity, Remark \ref{remark:attachconnected} shows $6$-connectedness, and Corollary \ref{cor:attachglobrigid} shows lack of generic global rigidity. Now, we present some notable examples of $G$. The graphs $K_n ~ {\cup}_{2,4} ~ C_{2,3,5,4}$, where $n \geq 7$, are GPR.  Note that for $K_6 ~ {\cup}_{2,4} ~ C_{2,3,5,4}$, the edges of $K_2$ and $K_4$ have no non-zero stresses, so it is not generically redundantly rigid.

Connelly \cite[8.3]{cn3} recently asked the following: if a graph $G$ is $(d+1)$-connected, GRR, and contains $K_{d+1}$ as a subgraph, is its Tutte realization necessarily infinitesimally rigid?  The concept of Tutte realizations is outside of the scope of this paper.  However, Connelly notes that an affirmative answer would imply that $G$ is always GGR.  The question is answered in the negative, considering $K_n ~ {\cup}_{2,4} ~ C_{2,3,5,4}$, where $n \geq 7$.

He also asked if, in any fixed dimension $d$, there are infinitely many GPR graphs.  Using the attachment construction described, we can find infinitely many graphs in $\mathbb{R}^5$ which are GPR.  Moreover, by the process of coning \cite{cw}, it is possible to preserve generic partial rigidity in these graphs in higher dimensions.  So, for any $d \geq 5$, this question has been answered in the affirmative.  It is still unknown for $d = 3$ and $d=4$.

We can also let $G$ be a $3$-chain with $a_1 = 2$ and $a_3 = 4$. The $3$-chains $C_{2, k, 4}$ with $k \geq 16$ are equivalent to $K_{6, k}$, and can easily be shown to be GRR and 6-connected using the algorithms described earlier in this paper, or finding the space of stresses. This makes $C_{2, k, 4} ~ {\cup}_{2,4} ~ C_{2,3,5,4}$ GPR. This class of graphs is especially notable since they form a $5$-ring, that is, a graph made from a $5$-chain by adding all edges between $A_1$ and $A_5$. It is intriguing that the size of one of the independent sets can be arbitrarily large. Also, a $5$-ring cannot be expressed as the subgraph of a complete bipartite graph.  We also remark that it is possible to have $G$ be some $4$-chain, creating a $6$-ring, which can be expressed as a subgraph of a complete bipartite graph.

Using Gortler, Healy and Thurston's algorithm \cite{ght}, we have proven that $C_{2, 15, 4} ~ {\cup}_{2,4} ~ C_{2,3,5,4}$ is GPR as well. However, $C_{2, 15, 4}$ is $K_{6, 15}$ and is not GRR. Knowing that $G$ is not GRR is not enough to say whether $G ~ {\cup}_{2,4} ~ C_{2,3,5,4}$ is GPR. Its properties rely on the individual characteristics of $G$.  We have seen an example ($K_6~ {\cup}_{2,4} ~ C_{2,3,5,4}$) that is not GPR, and another ($C_{2, 15, 4}~ {\cup}_{2,4} ~ C_{2,3,5,4}$) that is.

In addition to $C_{2,3,5,4}$, the $4$-chain $C_{3, 4, 5, 3}$ in $\mathbb{R}^5$ can also act as an attachment. The proof is analogous. There are more $4$-chains and also greater $k$-chains that we found in higher dimensions, but we have not found any $3$-chains, or any $k$-chains in lower dimensions. We have not completely categorized which $k$-chains can act as attachments. 

It is unknown whether graphs other than $k$-chains can act as attachments.  However, it would be extremely interesting to find a graph that could act as an attachment in $\mathbb{R}^3$, as right now there is only one GPR example in $\mathbb{R}^3$, and such an example would generate infinitely many graphs that are GPR.  We leave this question as an open problem.

\end{section}
\begin{section}{Further Exploration}
The $k$-chains with $\binom{d+2}{2}$ vertices have been fully explored in this paper. Additionally, for $v < \binom{d+2}{2}$,  simple calculations show that the $k$-chain is a subgraph of a complete bipartite graph that is not GLR.  Experimental evidence suggests the following conjecture for $v > \binom{d+2}{2}$:

\begin{conjecture}
Any $(d+1)$-connected $k$-chain in $\mathbb{R}^d$ with more than $\binom{d+2}{2}$ vertices is generically globally rigid.
\end{conjecture}

\begin{figure}
\includegraphics[scale=1]{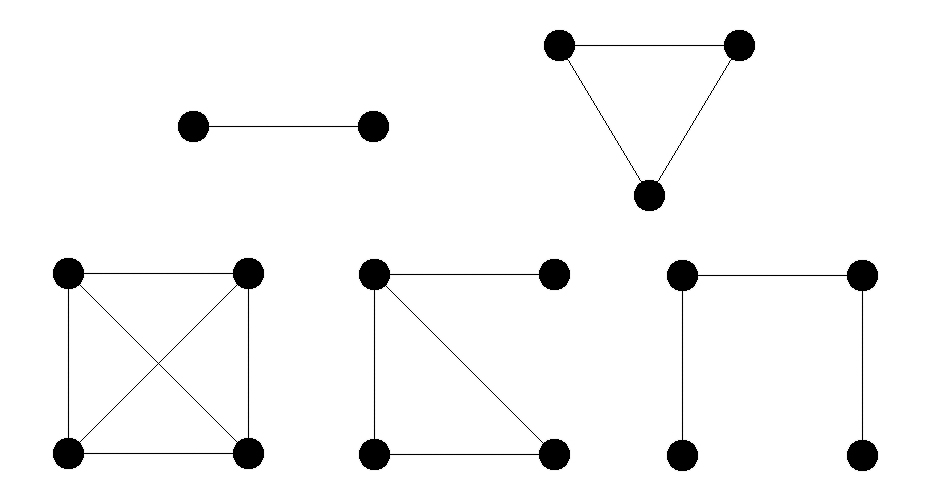}
\caption{Every irreducible graph with at most $4$ vertices.}
\label{fig:graphs}
\end{figure}

We have found a class of GPR subgraphs of GPR complete bipartite graphs. The more general question is to characterize which subgraphs of complete bipartite graphs are GPR. This is a very difficult question to answer generally, as we have also found examples of GPR graphs that are subgraphs of non-GPR complete bipartite graphs, as evidenced by the $6$-rings. 

The $k$-chains and $k$-rings that we have found can be characterized as part of a larger family of graphs. Given some initial connected graph $G$, replace each of the vertices with independent sets and completely connect the new vertices according to $G$.  When will this produce a graph that is GPR?

There exist many congruences among the initial graphs $G$. If there are two vertices in $G$ that connect to the exact same set of vertices, then they can be combined into one independent set. Call a graph $G$  \emph{irreducible} if there do not exist vertices that can be combined this way. Using this fact, we have identified $1$ irreducible connected graph with $2$ vertices, $1$ with $3$ vertices, $3$ with $4$ vertices, and $11$ with $5$ vertices [Figure \ref{fig:graphs}]. From there the number seems to grow exponentially. Experimentally, we have found GPR graphs made from every irreducible graph with at most $5$ vertices. We have also proved that we can make a GPR graph from every $k$-chain and $k$-ring with $k \geq 2$. Hence we suggest the following bold conjecture.

\begin{conjecture}
For any connected graph $G$ with $v > 1$, there exists some $a_1, a_2, \ldots, a_v$ and some $d$ such that if we replace each $v_i$ with an independent set of size $a_i$ and connect them accordingly, the resulting graph is GPR in $\mathbb{R}^d$.
\end{conjecture}

Remember that coning a graph that is GPR in $\mathbb{R}^d$ creates a graph that is GPR in $\mathbb{R}^{d+1}$. For virtually all of the initial graphs with $4$ or $5$ vertices, the GPR graph was obtained from a previous graph that was GPR, either by coning or by coning and removing some edges.  This may help to explain why the conjecture might be true. On the other hand, the $k$-chains and $k$-rings which are GPR are not obtained by coning, so there might be other types of graphs that resist coning.

\end{section}

\end{document}